\documentclass[a4paper,12pt]{article}
\usepackage[utf8]{inputenc}
 \usepackage{amssymb,amsmath,amsthm}

\newtheorem{theorem}{Theorem}
\newtheorem{lemma}{Lemma}
\newtheorem{conjecture}{Conjecture}

\title{Lower-bounds on the growth of power-free languages 
over large alphabets
}
 \author{Matthieu Rosenfeld\thanks{CNRS, LIS, Aix Marseille Universit\'e, Universit\'e de Toulon, Marseille, France\newline
 Supported by the ANR project CoCoGro (ANR-16-CE40-0005)}}
 
\begin{document}

%
%

\maketitle

\begin{abstract} 
We study the growth rate of some power-free languages.
For any integer $k$ and real $\beta>1$, we let $\alpha(k,\beta)$ be the growth rate of the number of $\beta$-free words of a given length over the alphabet $\{1,2,\ldots, k\}$. 
Shur studied the asymptotic behavior of $\alpha(k,\beta)$ for $\beta\ge2$ as $k$ goes to infinity. He suggested a conjecture regarding the asymptotic behavior of $\alpha(k,\beta)$ as $k$ goes to infinity when $1<\beta<2$. He showed that for $\frac{9}{8}\le\beta<2$ the asymptotic upper-bound holds.

We show that the asymptotic lower bound of his conjecture holds.  This implies that the conjecture is true for $\frac{9}{8}\le\beta<2$.
\end{abstract}

\section{Introduction}
A \emph{square} is a word of the form $uu$ where $u$ is a non-empty word.
We say that a word is \emph{square-free} (or avoids squares) if none of its factors is a square.
For instance, \verb!hotshots! is a square while \verb!minimize! is square-free.
In 1906, Thue showed that there are arbitrarily long ternary square-free words \cite{Thue06}.
This result is  often regarded as the starting point of combinatorics on words and the generalizations of this particular question received a lot of attention.

One such generalization is the notion of fractional power. 
A word of the form $w=xx\ldots xy$ where $x$ is non-empty and $y$ is a prefix of $x$ is a \emph{power of exponent $\frac{|w|}{|x|}$} and of \emph{period $|x|$} (we also say that $w$ is a \emph{$\left(\frac{|w|}{|x|}\right)$-power}).
Any square is a $2$-power.
For any real $\beta>1$ and word $w$, we say that $w$ is \emph{$\beta$-free} (resp.  \emph{$\beta^+$-free}) if it contains no factor that is an $\alpha$-power with $\alpha\ge\beta$ (resp. $\alpha>\beta$).
This notion was introduced by Dejean and received a lot of attention. Dejean's conjectured that for any $k>5$
there exists a $\left(\frac{k}{k-1}^+\right)$-free word over $k$ letters, but no $\left(\frac{k}{k-1}\right)$-free word \cite{Dejeanconj}.
After more than 30 years and the work of numerous authors the conjecture became a theorem in 2009 when the remaining cases were solved  independently by Currie and Rampersad and by Rao \cite{currieDejean,RaoDejean}.

The \emph{growth} (or \emph{growth rate}) of any language $L$ over an alphabet $\mathcal{A}$ is the quantity
$\lim_{n\rightarrow \infty} \left|L\cap\mathcal{A}^n\right|^{1/n}$.
It is a simple consequence of Fekete's Lemma that this quantity is well defined for any factorial language (i.e., a language that contains all factors of each of its elements).
The growth of languages avoiding some kind of forbidden patterns have also been studied a lot. 
It gives more information regarding how easily one can avoid these patterns.
Naturally, the growth rate of languages avoiding fractional repetitions received some attention (see \cite{shursurvey} for a survey on this topic). 
In particular, Shur studied the growth of $\beta$-free and $\beta^+$-free languages when the size of the alphabet is large \cite{shurquestion}.
For any $k$ and real $\beta>1$, we let $\alpha(k,\beta)$ be the growth rate of the set of $\beta$-free words.
Shur provided tight asymptotic formulas for $\alpha(k,\beta)$ for all $\beta\ge2$ as $k$ goes to infinity. 
However, he left the case $\beta<2$ open and gave the following conjecture.
\begin{conjecture}[\cite{shurquestion,shursurvey}]\label{mainconj}
 For any fixed integer $n\ge3$ and arbitrarily large integer $k$ the following holds
 \begin{equation}\label{eqsansplus}
\alpha\left(k,\frac{n}{n-1}\right)=k+1-n-\frac{n-1}{k}+O\left(\frac{1}{k^2}\right)
   \end{equation}
 \begin{equation}\label{eqavecplus}
\alpha\left(k,\frac{n}{n-1}^+\right)=k+2-n-\frac{n-1}{k}+O\left(\frac{1}{k^2}\right)
   \end{equation}
\end{conjecture}

Let us extend the strict total order $<$ of the reals to numbers of the form $x^+$ where $x$ is a real in such a way that $x^+$ is right after $x$ in the ordering for any real $x$. That is for all $x,y\in \mathbb{R}$, 
\begin{itemize}
 \item $x<y^+$ if and only if $x\le y$,
 \item and $x^+<y$ if and only if $x< y$.
\end{itemize}
With this definition $\alpha\left(k,x\right)$ is a decreasing function of $x$. Moreover, if conjecture \ref{mainconj} holds then for every integers $n$ and $k$ we have 
\begin{equation}\label{smallvar}
\alpha\left(k,\frac{n}{n-1}\right)-\alpha\left(k,\frac{n+1}{n}^+\right)=\frac{1}{k}+O\left(\frac{1}{k^2}\right)
\end{equation} and
\begin{equation}\label{bigjump}
 \alpha\left(k,\frac{n}{n-1}^+\right)-\alpha\left(k,\frac{n}{n-1}\right)=1+O\left(\frac{1}{k^2}\right)
\end{equation}

Hence, if the conjecture holds, it provides bounds on the asymptotic behavior of $\alpha\left(k,\beta\right)$ tight up to $\frac{1}{k}$ for every $\beta<2$. In particular, it implies that most of the jump between 
$\alpha\left(k,\frac{n}{n-1}\right)$ and $\alpha\left(k,\frac{n+1}{n}\right)$ occurs between $\alpha\left(k,\frac{n}{n-1}\right)$ and $\alpha\left(k,\frac{n}{n-1}^+\right)$. 
This conjecture implies other similar empirical facts that also hold for $\beta>2$ and illustrate the particular behavior of $\alpha(k,\beta)$ (facts \eqref{smallvar} and \eqref{bigjump} are respectively called \emph{small variation} and \emph{big jump} in  \cite{shurquestion}).

Shur showed that for any integer $n\le9$ the right-hand sides of equations \eqref{eqsansplus} and \eqref{eqavecplus} are indeed upper-bounds of the left-hand sides in both of these equations. 
In this article, we show that, for any integer $n>2$, the right-hand sides are lower bounds in both of these equations. This implies, in particular, that the conjecture holds for any integer $n\le9$ which provides tight bounds on the asymptotic behavior of $\alpha(k,\beta)$ for any $\beta$ such that $\frac{9}{8}\le\beta<2$.

The idea of the proof is in fact really simple and uses the idea that was introduced in \cite{prooftechnique}. To show that the language has exponential growth $\gamma$, we show the slightly stronger fact that for any $n$, the number of words of length $n+1$ is at least $\gamma$  times larger than the number of words of length $n$. The proof is a simple induction and exploits the locality of the problem to obtain a lower bound on the number of words of length $n+1$ based on the number of shorter words in the language. 
In this setting the same result could be obtained with the power series method for pattern avoidance \cite{BELL20071295,BLANCHETSADRI201317,doublepat,rampersadpowerseries}
, but the proof is slightly more complicated.

\section{The lower bounds}
We show the following result.
\begin{theorem}\label{mainth}
 For any fixed integer $n\ge2$ and arbitrarily large integer $k$ the following holds
 \begin{equation}\label{ineqsansplus}
\alpha\left(k,\frac{n}{n-1}\right)\ge k+1-n-\frac{n-1}{k}+O\left(\frac{1}{k^2}\right)\,,
   \end{equation}
 \begin{equation}\label{ineqavecplus}
\alpha\left(k,\frac{n}{n-1}^+\right)\ge k+2-n-\frac{n-1}{k}+O\left(\frac{1}{k^2}\right)\,.
   \end{equation}
\end{theorem}

\subsection{The first lower bound}
This subsection is devoted to the proof of equation \eqref{ineqsansplus}.
Let us first show the following stronger result.
\begin{lemma}\label{lowboundlemma}
Let $k$ and $n$ be two integers with $k>n>1$. For all $i$, let $C_i$
be the number of $\left(\frac{n}{n-1}\right)$-free words of length $i$ over a $k$-letter alphabet.
If $x>1$ is a real such that we have $k-(n-1)\frac{x}{x-1}\ge x$, then for any integer $i$
$$C_{i+1}\ge x C_{i} \,.$$ 
\end{lemma}
\begin{proof}
 We proceed by induction on $i$. 
By definition, we have $C_1=k$ and $C_2= k(k-1)$ and by assumptions, we have $x\le k-(n-1)\frac{x}{x-1}\le k-1$ which implies $C_2\ge xC_1$.
 Let $i$ be an integer such that the Lemma holds for any integer smaller than $i$.
 Let $F$ be the set of words of length $i+1$ that are not $\left(\frac{n}{n-1}\right)$-free but whose prefix of length $i$ is $\left(\frac{n}{n-1}\right)$-free.
 Then
 \begin{equation}\label{maineqF}
  C_{i+1}=kC_i-|F|
 \end{equation}
We now bound the size of $F$. 
For every $j$, let $F_j$ be the set of words from $F$ that contains a repetition of period $j$ and exponent at least $\frac{n}{n-1}$. Then $|F|\le \sum_{j\ge1}|F_j|$.

For any word $w\in F_j$, there exist $x$ and $y$ such that  that $xy$ is a suffix of $w$ and is a repetition of period $x$ with $|x|=j$ and of exponent at least $\frac{n}{n-1}$ which implies $|y|\ge\frac{j}{n-1}$.
Moreover, if we remove the last letter of $xy$ we obtain a $\left(\frac{n}{n-1}\right)$-free word which implies that
$|y|-1<\frac{j}{n-1}$ and thus $|y|=\left\lceil\frac{j}{n-1}\right\rceil$. Since $xy$ is a repetition of period $x$ it also implies that $y$ is uniquely determined by $x$.
Thus, for any word $w\in F_j$ the last $\left\lceil\frac{j}{n-1}\right\rceil$ letters are uniquely determined by the prefix of length $i+1-\left\lceil\frac{j}{n-1}\right\rceil$ of $w$. The prefix of length $i+1-\left\lceil\frac{j}{n-1}\right\rceil$ of any such word  belongs to $C_{i+1-\left\lceil\frac{j}{n-1}\right\rceil}$ since it is $\left(\frac{n}{n-1}\right)$-free. We deduce the following bound
$$|F_j|\le C_{i+1-\left\lceil\frac{j}{n-1}\right\rceil}\,.$$
By the induction hypothesis, we get 
$$|F_j|\le x^{1-\left\lceil\frac{j}{n-1}\right\rceil}C_{i}\,.$$
Thus
$$ |F|\le \sum_{j\ge1} x^{1-\left\lceil\frac{j}{n-1}\right\rceil}C_{i}
 = (n-1)C_{i}\sum_{j\ge1} x^{1-j}=(n-1)C_i\frac{x}{x-1}\,.
$$
Substituting $|F|$ in equation \eqref{maineqF} yields
$$C_{i+1}\ge C_i\left(k-(n-1)\frac{x}{x-1}\right)\ge C_i x$$
as desired.\qed
\end{proof}

We can now easily deduce equation \eqref{ineqsansplus}.

\begin{lemma}\label{taylorexpansion}
 For any fixed integer $n$ and arbitrarily large integer $k$ the following holds
 $$\alpha\left(k,\frac{n}{n-1}\right)\ge k+1-n-\frac{n-1}{k}+O\left(\frac{1}{k^2}\right)\,.$$ 
\end{lemma}
\begin{proof}
By Lemma \ref{lowboundlemma}, we know that for any integers $k$ and $n$, and real  $x>1$ such that we have 
\begin{equation}\label{secdeg}
 k-(n-1)\frac{x}{x-1}\ge x
\end{equation}
we also have
$$\alpha\left(k,\frac{n}{n-1}\right)\ge x\,.$$
Since equation \eqref{secdeg} is a 2nd degree equation it is easy to see that if $k \ge n+2\sqrt{n+1}$, then
$x= \frac{k+2 -n + \sqrt{4 + (k - n)^2 - 4 n}}{2}$ is the largest solution of equation \eqref{secdeg}. 
Thus as long as $k \ge n+2\sqrt{n+1}$, we have
$$\alpha\left(k,\frac{n}{n-1}\right)\ge \frac{k+2 -n + \sqrt{4 + (k - n)^2 - 4 n}}{2}\,.$$
Let $f$ be the function that maps any real $y$ to $$f(y)=\frac{1+2y -ny + \sqrt{4y^2 + (1 - ny)^2 - 4 n y^2}}{2}\,,$$ then 
$$\alpha\left(k,\frac{n}{n-1}\right)\ge k \times f\left(\frac{1}{k}\right)\,.$$ 
The first terms of the Taylor Series of $f$ at $0$ are
$$f(y)=1+(1-n) y+(1-n) y^2+O\left(y^3\right)$$ and we easily deduce that for abritrarily large $k$ 
 $$\alpha\left(k,\frac{n}{n-1}\right)\ge k+1-n-\frac{n-1}{k}+O\left(\frac{1}{k^2}\right)$$ 
as desired.\qed
\end{proof}

\subsection{The second lower bound}

This subsection is devoted to the proof of equation \eqref{ineqavecplus}.
The proof is almost the same as the proof of Lemma \ref{lowboundlemma}.
The only difference is that we get
$|y|=\left\lfloor\frac{j}{n-1}\right\rfloor+1$ instead of 
$|y|=\left\lceil\frac{j}{n-1}\right\rceil$ which impacts the computations. We still provide the full proof for the sake of completeness.
\begin{lemma}\label{lowboundlemma2} 
Let $k$ and $n$ be two integers with $k>n>1$. For all $i$, let $C_i$
be the number of $\left(\frac{n}{n-1}^+\right)$-free words of length $i$ over a $k$-letter alphabet. If $x>1$ is a real such that we have $k+1-(n-1)\frac{x}{x-1}\ge x$, then for any integer $i$
$$C_{i+1}\ge x C_{i} \,.$$ 
\end{lemma}
\begin{proof}
We proceed by induction on $i$. 
 Let $i$ be an integer such that the Lemma holds for any integer smaller than $i$.
 Let $F$ be the set of words of length $i+1$ that are not $\left(\frac{n}{n-1}^+\right)$-free but whose prefix of length $i$ is $\left(\frac{n}{n-1}^+\right)$-free.
 Then
 \begin{equation}\label{maineqF2}
  C_{i+1}=kC_i-|F|\,.
 \end{equation}
We now bound the size of $F$. 
For every $j$, let $F_j$ be the set of words from $F$ that contains a repetition of period $j$ and exponent greater than $\frac{n}{n-1}$. Then clearly $|F|\le \sum_{j\ge1}|F_j|$.

By definition, for any word $w\in F_j$, there exist $u,x$ and $y$ such that $|x|=j$, $y$ is a prefix of $x$, $|y|>\frac{j}{n-1}$ and $w=uxy$. 
Moreover, if we remove the last letter of $xy$ we obtain a $\left(\frac{n}{n-1}^+\right)$-free word which implies that
$|y|-1\le\frac{j}{n-1}$ and thus $|y|=\left\lfloor\frac{j}{n-1}\right\rfloor+1$. 
Thus for any word $w\in F_j$ the last $\left\lfloor\frac{j}{n-1}\right\rfloor+1$ letters are uniquely determined by the prefix of length $i-\left\lfloor\frac{j}{n-1}\right\rfloor$ of $w$. By definition, the prefix of length $i-\left\lfloor\frac{j}{n-1}\right\rfloor$ of any such word is $\left(\frac{n}{n-1}^+\right)$-free and belongs to $C_{i-\left\lfloor\frac{j}{n-1}\right\rfloor}$. This implies the following bound
$$|F_j|\le C_{i-\left\lfloor\frac{j}{n-1}\right\rfloor}\,.$$
By the induction hypothesis, we get 
$$|F_j|\le x^{-\left\lfloor\frac{j}{n-1}\right\rfloor}C_{i}\,.$$
Thus
$$ |F|\le C_{i}\sum_{j\ge1} x^{-\left\lfloor\frac{j}{n-1}\right\rfloor}
 = C_{i}\left(-1+(n-1)\sum_{j\ge0} x^{-j}\right)=C_i\left(\frac{(n-1)x}{x-1}-1\right)\,.
$$
Substituting $|F|$ in equation \eqref{maineqF2} yields
$$C_{i+1}\ge C_i\left(k+1-\frac{(n-1)x}{x-1}\right)\ge C_i x$$
as desired.\qed
\end{proof}
 The condition is once again a quadratic inequality so we easily verify that the condition holds for
 $$x=\frac{k+3 -n + \sqrt{5 + 2 k + k^2 - 2 (3 + k) n + n^2}}{2}\,.$$
 We can compute the first terms of a well chosen Taylor Series to obtain the following result (we can also simply ask Mathematica or any other formal mathematical software the asymptotic behavior of this function). 
\begin{lemma}
 For any fixed integer $n$ and arbitrarily large integer $k$ the following holds
 $$\alpha\left(k,\frac{n}{n-1}\right)\ge k+2-n-\frac{n-1}{k}+O\left(\frac{1}{k^2}\right)\,.$$ 
\end{lemma}

\section{Conclusion}
Let us insist on the fact that our proof is very simple. The main argument is less than a page long and the more advanced mathematics are geometric series (if we ignore the computation of the Taylor polynomial, which is not needed). However, for $\beta\ge2$ this approach does not provide lower bounds of $\alpha\left(k,\beta\right)$ as tight as the bounds from \cite{shurquestion}.
We believe that Conjecture \ref{mainconj} holds and we were able to make some progress in that direction.

Let us call the word obtained by erasing the first period of a repetition the \emph{tail} of the repetition and let $\alpha'\left(k,\beta\right)$ be the growth of the language of the words that contains no $\beta$-power of tail of length at most $2$. It is probably the case that the following stronger conjecture holds 
\begin{conjecture}\label{strongconj}
 For any fixed integer $n\ge2$ and arbitrarily large integer $k$ the following holds
 \begin{equation}
\alpha'\left(k,\frac{n}{n-1}\right)=k+1-n-\frac{n-1}{k}+O\left(\frac{1}{k^2}\right)
   \end{equation}
 \begin{equation}
\alpha'\left(k,\frac{n}{n-1}^+\right)=k+2-n-\frac{n-1}{k}+O\left(\frac{1}{k^2}\right)
   \end{equation}
\end{conjecture}
That is, the coefficient of the term $\frac{1}{k}$ is probably dictated by the repetitions of tail of length at most $2$. It might even be true that the coefficient of the term  $\frac{1}{k^j}$ is dictated by the repetitions of tail of length at most $j+1$. This idea was already discussed in more details in \cite[Section 5]{shurquestion}.

Let use finally recall that the result that we showed is in fact slightly stronger since for $k$ large enough, we have
$$\alpha\left(k,\frac{n}{n-1}\right)\ge \frac{k+2 -n + \sqrt{4 + (k - n)^2 - 4 n}}{2}\,.$$
Conjectures \ref{mainconj} and \ref{strongconj} both imply that this bound is tight up to $O\left(\frac{1}{k^2}\right)$, but this bound might be tighter than that. The same might also be true for the bound on $\alpha\left(k,\frac{n}{n-1}^+\right)$. For $n=2$ these lower bounds can be compared to Theorem 2 of \cite{shurquestion} and in this case our lower bounds are only tight up to $O\left(\frac{1}{k^2}\right)$.


\begin{thebibliography}{}
\bibitem{BELL20071295}
J.~P. Bell and T.~L. Goh.
\newblock Exponential lower bounds for the number of words of uniform length
  avoiding a pattern.
\newblock {\em Information and Computation}, 205(9):1295--1306, 2007.

\bibitem{BLANCHETSADRI201317}
F. Blanchet-Sadri and B. Woodhouse.
\newblock Strict bounds for pattern avoidance.
\newblock {\em Theoretical Computer Science}, 506:17--28, 2013.

\bibitem{currieDejean}
J.D. Currie, N. Rampersad. 
\newblock A proof of Dejean’s conjecture. 
\newblock {\em Math. Comput.} \textbf{80}, 1063--1070 (2011)

\bibitem{Dejeanconj} F. Dejean. 
\newblock Sur un theoreme de Thue.
\newblock {\em J. Comb. Theory, Ser. A}\textbf{13}(1), 90--99 (1972)

\bibitem{doublepat}
P. Ochem.
\newblock Doubled patterns are 3-avoidable.
\newblock {\em Electronic Journal of Combinatorics}, 23(1), 2016.

\bibitem{rampersadpowerseries}
N. Rampersad.
\newblock Further applications of a power series method for pattern avoidance.
\newblock {\em Electronic Journal of Combinatorics}, 18:134, 2011.

\bibitem{RaoDejean}
M. Rao.
\newblock Last cases of Dejean’s conjecture.
\newblock {\em Theor. Comput. Sci.} \textbf{412}, 3010--3018 (2011)

\bibitem{prooftechnique}
M. Rosenfeld.
\newblock Another approach to non-repetitive colorings of graphs of bounded degree. 
\newblock {\em Electronic Journal of Combinatorics}, 27(3), 2020.

\bibitem{shursurvey}
A.M. Shur.
\newblock Growth Properties of Power-Free Languages. 
\newblock {\em Computer Science Review} 6(5–6), 187--208 (2012)
,

\bibitem{shurquestion}
A.M. Shur.
\newblock Growth of Power-Free Languages over Large Alphabets. 
\newblock {\em Theor. Comp. Sys.} 54, 224--243. (2014)

\bibitem{Thue06}
 A. Thue.
\newblock Über unendliche Zeichenreihen.
\newblock {\em Kra. Vidensk. Selsk. Skrifter. I. Mat.-Nat. Kl.}, Christ.7,1--22 (1906)
\end{thebibliography}
\end{document}